\documentclass[a4paper,11pt]{article}
\usepackage[hidelinks]{hyperref}
\hypersetup{
	colorlinks=true,   
	linktoc=all,        
	linkcolor=red,     
	citecolor=blue,}     
\usepackage{lineno}
\usepackage{tikz}
\usepackage{graphicx}
\usepackage{subfigure}
\usepackage{amssymb}
\usepackage{latexsym,bm}
\usepackage{multicol}
\usepackage{indentfirst}
\usepackage{amssymb,amsfonts}
\usepackage{amsmath}
\usepackage{amsthm}
\usepackage{setspace}
\newcommand{\ignore}[1]{}

\newtheorem{theorem}{Theorem}[section]
\newtheorem{lemma}[theorem]{Lemma}
\newtheorem{claim}[theorem]{Claim}
\newtheorem{fact}[theorem]{Fact}
\newtheorem{corollary}[theorem]{Corollary}

\newtheorem{conjecture}[theorem]{Conjecture}

\newtheorem{question}[theorem]{Question}

\newtheorem{observation}[theorem]{Observation}
\begin{document}
	\date{}
	\begin{spacing}{1.03}
		\title{A Note on Generalized Erd\H{o}s–Rogers Problems}
		\author{Longma Du,\footnote{School of Mathematics, Shandong University, Jinan, 250100, P.~R.~China. Email: {\tt 202520303@mail.sdu.edu.cn}. }
       \;\; \;  Xinyu Hu,\footnote{Data Science Institute, Shandong University, Jinan, 250100, P.~R.~China. Email: {\tt huxinyu@sdu.edu.cn}. Supported by National Postdoctoral Fellowship Program (C-tier) (GZC20252005).}
			\;\; \; Ruilong Liu,\footnote{School of Mathematics, Shandong University, Jinan, 250100, P.~R.~China. Email: {\tt liuruilong@mail.sdu.edu.cn}. }
\;\; \; Guanghui Wang \footnote{
School of Mathematics, Shandong University, Jinan, 250100, P. R. China. Email: {\tt ghwang@sdu.edu.cn}. Supported in part  by State Key Laboratory of Cryptography and Digital Economy Security.}
}
\maketitle
\begin{abstract}
For a $k$-uniform hypergraph $F$ and positive integers $s$ and $N$, the generalized Erd\H{o}s-Rogers function $f^{(k)}_{F,s}(N)$ denotes the largest integer $m$ such that every $K_s^{(k)}$-free $k$-graph on $N$ vertices contains an $F$-free induced subgraph on $m$ vertices. In particular, if $F = K^{(k)}_t$, then we write $f^{(k)}_{t,s}(N)$ for $f^{(k)}_{F,s}(N)$.
Mubayi and Suk (\emph{J. London. Math. Soc. 2018}) conjectured that $f^{(4)}_{5,6}(N)=(\log \log N)^{\Theta(1)}$. 
Motivated by this conjecture, we prove that
$f^{(4)}_{5^{-},6}(N)=(\log\log N)^{\Theta(1)}$,
where $5^{-}$ denotes the $4$-graph obtained from $K_5^{(4)}$ by deleting one edge. Our proof combines a probabilistic construction of a $2$-coloring of pairs with a stepping-up construction and an analysis of multi-layer local extremum structures. 
Furthermore, we derive an upper bound for a more general Erdős–Rogers function, which implies the lower bound $r_4(6,n)\ge 2^{2^{cn^{1/2}}}$.
By applying a variant of the Erd\H{o}s–Hajnal stepping-up lemma due to Mubayi and Suk, we also slightly improve the lower bound for $r_k(k+2,n)$. 
\medskip


\end{abstract}

\section{Introduction}
A $k$-uniform hypergraph $H$ ($k$-graph for short) consists of a vertex set $V(H)$ and a collection of $k$-element subsets of $V(H)$. Let $K_n^{(k)}$ be the complete $k$-graph on $n$ vertices. 
The classical off-diagonal Ramsey number $r_k(s,n)$ is the minimum $N$ such that every red/blue coloring of the edges of $K_N^{(k)}$ results in a monochromatic red copy of $K^{(k)}_s$ or a monochromatic blue copy of $K^{(k)}_n$, where $k,s$ are fixed and $n$ tends to infinity. 
Ramsey numbers have been extensively studied since Ramsey’s seminal work \cite{R}, with many classic results \cite{A-K-S-1,A-K-S-2,B-K-2,C-G-E,E-H-Con,E-R-2,K-1,L-R-Z,M-S-3,M-S-4,S-1}.
In recent years, there have been many breakthroughs on Ramsey numbers, especially in graphs \cite{C-J-M-S,C-F-S-3,H-H-K-P,H-M-S,M-S-X,M-V-2,C-G-M-S,G-N-N-W}. We refer the reader to Conlon, Fox and Sudakov \cite{C-F-S-4}, and Morris \cite{Mor} for two nice surveys on this topic.

Given a $k$-graph $F$, an $F$-independent set in a $k$-graph $H$ is a vertex subset $S$ such that $H[S]$ contains no copy of $F$. If $F=K^{(k)}_k$, then it is just an independent set. Let $\alpha_F(H)$ denote the size of the largest $F$-independent set in $H$. If $F=K^{(k)}_t$ with $t\ge k$, we use the simpler notation $\alpha_t(H)$. For two $k$-graphs $F$ and $G$ and a positive integer $N$, the
\emph{generalized Erd\H{o}s--Rogers} function $f^{(k)}_{F,G}(N)$ denotes
the largest $m$ such that every $G$-free $k$-graph on $N$ vertices
contains an $F$-free induced subgraph on $m$ vertices. Namely,
\[
f^{(k)}_{F,G}(N)=\min\{\alpha_F(H): |V(H)|=N \text{ and } H \text{ is } G\text{-free}\}.
\]
Clearly, $f^{(k)}_{F',G}(N)\le f^{(k)}_{F,G}(N) $ if $F'\subseteq F$ and $f^{(k)}_{F,G}(N)\le f^{(k)}_{F,G'}(N) $ if $G' \subseteq G$. When $G=K_s^{(k)}$, we abbreviate $f^{(k)}_{F,K_s^{(k)}}(N)$ by
$f^{(k)}_{F,s}(N)$. In particular, if $F=K_t^{(k)}$, then we write
$f^{(k)}_{t,s}(N)$.
\medskip

This function has been extensively studied over the past 60 years in the case where $G$ and $F$ are cliques and is known as the Erd\H{o}s-Rogers function (see, e.g. \cite{ak,bh,D-R-R,dr,E-R-1,GJ20,H-L,Janzer,kri-1,kri-2,Sud-1,Sud-2,W-1}). Recently, Balogh, Chen and Luo \cite{B-C-L-1}, Mubayi and Verstra\"{e}te \cite{M-V-1-1} and Gishboliner, Janzer and Sudakov \cite{G-J-S-1}, started the systematic study of the function $f^{(2)}_{F,G}$. Their results mostly concern the case when $G$ is a clique, and establish bounds for $f^{(2)}_{F,s}(N)$ when $F$ satisfies certain properties such as clique-free, bipartite, containing a cycle, or having large minimum degree. In \cite{H-N-1}, He and Nie considered the natural generalization of this line of research to hypergraphs, and obtained a sufficient condition for $f^{(k)}_{F,G}(N)$ to be polynomial. In this paper, we focus on the hypergraph setting, so we do not consider the graph case here. We refer the reader to \cite{C-F-S-4,M-S-1} for excellent surveys on the graph setting. 

The problem of determining $f^{(k)}_{k,s}(N)$ is equivalent to determining the Ramsey number $r_k(s,n)$. Formally,
\begin{align}\label{R-EG}
r_k(s,n)=\min\{N: f^{(k)}_{k,s}(N)\ge n\}.
\end{align}

Dudek and Mubayi \cite{D-M} proved that $\Omega((\log_{(k-2)} N)^{1/4+o(1)})\le f^{(k)}_{t,t+1}(N)\le O((\log N)^{1/(k-2)})$ by showing \begin{align}\label{ditui}
    f^{(k-1)}_{t-1,s-1}(\lfloor(\log N)^{1/(k-1)}\rfloor)\le f^{(k)}_{t,s}(N)
\end{align}
 for $3\le k\le t<s$, where $\log_{(0)}(N)=N$ and as usual, $\log_{(i+1)}N=\log (\log_{(i)}N)$. All logarithms are to base $2$ unless otherwise stated. An interesting question is to determine the exponent of $\log N$ in $f^{(3)}_{t,t+1}(N)$. Conlon, Fox and Sudakov \cite{C-F-S-2} improved this to $f^{(k)}_{t,t+1}(N)\ge \Omega((\log_{(k-2)} N)^{1/3+o(1)})$ by showing $f^{(3)}_{t,t+1}(N)\ge \Omega((\log N)^{1/3+o(1)})$, and they also proposed to close the gap between the upper and lower bounds. Motivated by this direction, Mubayi and Suk \cite[Theorem 3.2]{M-S-5} showed that $f^{(k)}_{k+1,k+2}(N)\le O(\log_{(k-13)} N)$ for $k\ge 14$. They proposed the following conjecture.
 \begin{align}\label{mubayiand suk}
     f^{(k)}_{k+1,k+2}(N)=(\log_{(k-2)}N)^{\Theta(1)}
 \end{align}
for each integer $k\ge 4$. This conjecture is equivalent to the statement that $f^{(k)}_{k+1,s}(N)=(\log_{(k-2)}N)^{\Theta(1)}$ for every $s\ge k+2$ once we note that $(\log_{(k-2)}N)^{\Omega(1)}\le f^{(k)}_{k+1,s}(N)\le f^{(k)}_{k+1,k+2}(N)$
for $k\ge 3$ and $s\ge k+2$. Here, the first inequality follows from (\ref{ditui}) and $f^{(2)}_{3,s}(N)\ge N^{1/(s-2)}$, which was shown by Bollob\'{a}s and Hind \cite{bh}. Recently, the authors \cite{D-H-L-W} settled the conjecture for all $s\ge k+7$, proving that $f^{(4)}_{5,s}(N)=(\log\log N)^{O(1)}$ for all $s\ge 11$ 
via a stepping-up lemma.

In \cite{F-H-L-L}, Fan, Hu, Lin and Lu showed that  $f^{(k)}_{k+1,k+2}(N)\le O( \log_{(k-3)} N)$ for $k \ge 5$.
Thus, only one iterated logarithm gap remains in proving (\ref{mubayiand suk}). Moreover, \cite[Theorem 4.5 and 4.6]{F-H-L-L} proved a variant of the Erd\H{o}s-Hajnal stepping-up lemma to obtain that in order to solve (\ref{mubayiand suk}), it suffices to show that
\begin{conjecture}[Mubayi and Suk \cite{M-S-5}]\label{con-1}
It holds that $f^{(4)}_{5,6}(N)=(\log \log N)^{\Theta(1)}.$
\end{conjecture}

Motivated by this conjecture, we consider the following related generalized Erd\H{o}s--Rogers function, which corresponds to a slightly weaker problem.

For $k \ge 4$ and $1 \le t \le k+1$, let $H_t^{k}$ denote the unique $k$-graph on $k+1$ vertices with $t$ edges. Equivalently, $H_t^k$ is obtained from $K_{k+1}^{(k)}$ by deleting $k+1-t$ edges.
Clearly, $\alpha_k(H)=\alpha_{H^{k}_1}(H)$ for any non-complete $k$-graph $H$. 
Thus, \begin{align}\label{k-H_k}
       f^{(k)}_{k,s}(N)=f^{(k)}_{{H^{k}_1},s}(N).
      \end{align} 
We write $H^{4}_4$ as $5^{-}$. Note that 
\begin{align}\label{lower bound of f}
    f^{(4)}_{5^{-},6}(N)\ge f^{(4)}_{{H^{4}_1},6}(N)\overset{(\ref{k-H_k})}{=}f^{(4)}_{4,6}(N)\ge (\log \log N)^{\Omega(1)},
\end{align}
where the last inequality holds since (\ref{R-EG}) and $r_4(s,n)\le2^{2^{n^{O(1)}}}$ for $s\ge 5$ as shown by Erd\H{o}s and Rado \cite{E-R-2}.

In this paper, we first use the idea in \cite{D-H-L-W} to show that the magnitude of (\ref{lower bound of f}) is sharp. Specifically, we provide a probabilistic method to find a red/blue coloring $\phi$ on the pairs of $[2^{\Theta(n)}]$ that satisfies a  certain ``good'' property (see Lemma \ref{phi-$H^4_4$-6}). Based on  $\phi$, 
we then construct a $4$-graph $H$ on  $2^{2^{\Theta(n)}}$ vertices that is $K^{(4)}_6$-free and satisfies $\alpha_{5^-}(H)<2^{5}n^{5}+1$. This construction relies on analyzing $5$-layer local maxima sequences in the stepping-up method and yields a new upper bound $f^{(4)}_{5^-,6}(N)\le (\log\log N)^{O(1)}$. Therefore, we can obtain the following theorem.

\begin{theorem}\label{center-1}
We have $f^{(4)}_{5^-,6}(N)= (\log\log N)^{\Theta(1)}$.
\end{theorem}

More generally, by applying a variant of the Erd\H{o}s-Hajnal stepping-up lemma by Fan, Hu, Lin and Lu \cite[Theorem 4.5 and 4.6]{F-H-L-L}, we can obtain the following corollary.
\begin{corollary}\label{coro-1}
For each $k\ge5$, we have $f^{(k)}_{H^{k}_4,k+2}(N)= (\log_{(k-2)} N)^{\Theta(1)}$. 
\end{corollary}

We now return to the off-diagonal Ramsey number $r_k(s,n)$. 
For $3$-graphs, Conlon, Fox and Sudakov \cite{C-F-S-3} obtained that for every $s\ge4$, $2^{\Omega(n\log n)}\le r_3(s,n)\le 2^{O(n^{s-2}\log n)}$, 
thereby improving the upper bound of Erdős and Rado \cite{E-R-2} and the lower bound of Erdős and Hajnal \cite{E-H-Con}. For $s>k\ge 4$, it is known that $r_k(s,n)<twr_{k-1} (n^{O(1)})$ \cite{E-R-2}, where the tower function $twr_{k}(x)$ is defined by $twr_{1}(x)=x$ and $twr_{i+1}(x)=2^{twr_{i}(x)}$. By applying the Erd\H{o}s-Hajnal stepping-up lemma in the off-diagonal setting, it follows that $r_k(s,n)\ge twr_{k-1} (\Omega(n))$, for $k\ge 4$ and for all $s\ge 2^{k-1}-k+3$. A fundamental and important conjecture about $r_k(s,n)$ was proposed by Erd\H{o}s and Hajnal \cite{E-H-Con}.
\begin{conjecture}[Erd\H{o}s and Hajnal \cite{E-H-Con}]\label{E-H-C}
    For fixed integers $4\le k<s$, it holds that
$r_k(s,n)\ge \operatorname{twr}_{k-1}(\Omega(n)).$
\end{conjecture} 
Erd\H{o}s and Hajnal (see \cite{G-R-S-1}) 
showed that $r_4(7,n)\ge 2^{2^{\Omega(n)}}$. In \cite{cfs}, Conlon, Fox and Sudakov modified the Erd\H{o}s-Hajnal stepping-up lemma to show that Conjecture \ref{E-H-C} holds for all $s\ge \lceil\frac{5k}{2}\rceil-3$. Mubayi and Suk \cite{M-S-4} as well as Conlon, Fox and Sudakov, independently verified Conjecture \ref{E-H-C} for $k\ge 4$ and $s\ge k+3$. However, determining the lower bounds for $r_k(k+1,n)$ and $r_k(k+2,n)$ is much more difficult. Mubayi and Suk \cite{M-S-3} established the lower bounds $r_k(k+1,n)\ge twr_{k-2} (n^{\Omega(\log n)})$ and $r_k(k+2,n)\ge twr_{k-1} ({\Omega(n^{1/5})})$, which represent the previous best results.

		We also construct, by the probabilistic method, a red/blue coloring $\phi$ of the pairs of $[2^{\Theta(n)}]$ with a certain ``good'' property (see Lemma \ref{phi-$H^4_2$}). Using $\phi$, 
		we then construct a $4$-graph $H$ on  $2^{2^{\Theta(n)}}$ vertices that is $K_6^{(4)}$-free and satisfies $\alpha_{H^4_2}(H)<n^2$.  This construction relies on an analysis of monotone sequences together with a greedy argument within the stepping-up method, and yields the upper bound $f^{(4)}_{H^4_2,6}(N)\le (\log\log N)^{2}$. Consequently, we obtain the following theorem.

\begin{theorem}\label{center-2}
    We have $f^{(4)}_{H^4_2,6}(N)\le O((\log\log N)^2)$.
\end{theorem}

\noindent\textit{Remark.} 
By the same construction as in the proof of Theorem~\ref{center-2}, together with a more delicate analysis, one also obtains $f^{(4)}_{H_3^4,6}(N)\le O((\log\log N)^3)$. Since the argument is entirely analogous to that of Theorem~\ref{center-2}, we omit the details.

Since $f^{(4)}_{H^4_2,6}(N)\ge f^{(4)}_{H^4_1,6}(N)\overset{(\ref{k-H_k})}{=}f^{(4)}_{4,6}(N)$, it follows from (\ref{R-EG}) that the following corollary holds.

\begin{corollary}\label{coro-2}
    For all $n\ge 5$, we have $r_4(6,n)\ge 2^{2^{cn^{1/2}}}$, where $c>0$ is an absolute constant.
\end{corollary}

In \cite{M-S-3}, Mubayi and Suk produced a variant of the classic Erd\H{o}s-Hajnal stepping-up lemma, namely, they showed that for $k\ge 5$ and $n\ge s\ge k+1$, $r_k(s,2kn)>2^{r_{k-1}(s-1,n)-1}$. Applying this stepping-up lemma together with Corollary \ref{coro-2}, we obtain 
$r_k(k+2,n)\ge twr_{k-1}(\Omega(n^{1/2}))$, which slightly improves the previous best lower bound $twr_{k-1}(\Omega(n^{1/5}))$ of  Mubayi and Suk.

Beyond these results, it would be interesting to refine the methods developed here to address the following question.

\begin{question}
    Is it true that $f^{(4)}_{5^-,5}(N)=(\log\log N)^{O(1)}$?
\end{question}
\noindent

A positive answer would imply the conjectured double-exponential lower bound for $r_4(5,n)$, which, combined with the stepping-up lemma, would then prove the conjecture of Erd\H{o}s and Hajnal. It is the last remaining open case concerning the tower height for classical off-diagonal hypergraph Ramsey numbers.

\section{Properties of the stepping-up technique}\label{sec2}
In this paper, we will apply several variants of the Erd\H{o}s-Hajnal stepping-up lemma. We shall use the following notations and definitions unless otherwise stated.

Given some natural number $D$, let $V = \{0, 1, \ldots, 2^D - 1\}$. Then for any $v \in V$, write
$v = \sum_{i=0}^{D-1} v(i)2^i$ where  $v(i) \in \{0, 1\}$ for each  $i$.
For any $u \neq v$, let $\delta(u, v)$ denote the largest $i \in \{0, 1, \ldots, D - 1\}$ such that $u(i) \neq v(i)$. We always use $\langle v_1, v_2,\ldots, v_r\rangle$ to denote an ordered set with $v_1< v_2<\cdots<v_r$ under the order relation. Given any set $S=\langle v_1, v_2,\ldots, v_r\rangle$, we always write for $i\in[r-1]$, $\delta_i=\delta(v_i,v_{i+1})$ and $\delta(S)=(\delta_1,\ldots,\delta_{r-1})=(\delta_i)_{i=1}^{r-1}$. Sometimes we write $(\delta_i)_{i=1}^{r-1}$ as $\{\delta_i\}_{i=1}^{r-1}$. For convenience, if inequalities are known between consecutive $\delta$’s, this will be indicated in the sequence by replacing the comma with the respective sign. For instance, assume that $S=\langle v_1,\ldots, v_5\rangle$ and $\delta_1< \delta_2 >\delta_3< \delta_4$. Then since $\delta(v_1,v_2,v_3,v_4)=(\delta_1,\delta_2,\delta_3)$
has $\delta_1< \delta_2 >\delta_3$, we will write
$\delta(v_1,v_2,v_3,v_4)=(\delta_1<\delta_2>\delta_3)$.
Similarly, if not all inequalities are known, as in $\delta(v_1,v_2,v_4,v_5)$, we write $\delta(v_1,v_2,v_4,v_5)=(\delta_1<\delta_2~,~\delta_4)$.

We say that $\delta_i$ is a \emph{local minimum} if
$\delta_{i-1}>\delta_i<\delta_{i+1}$, a \emph{local maximum} if $\delta_{i-1}<\delta_i>\delta_{i+1}$, and a \emph{local extremum} if it is either a local minimum or a local maximum. We call $\delta_i$ a \emph{local monotone} if $\delta_{i-1}<\delta_i<\delta_{i+1}$ or $\delta_{i-1}>\delta_i>\delta_{i+1}$. We say $\delta_1,\ldots,\delta_{r-1}$
form a monotone sequence if $\delta_1<\cdots<\delta_{r-1}$ (monotone increasing) or $\delta_1>\cdots>\delta_{r-1}$
(monotone decreasing), i.e., there is no local extremum.

We then have the following stepping-up properties, see \cite{G-R-S-1}.\medskip

\textbf{Property I.} For every triple $u < v < w$, $\delta(u,v) \neq \delta(v,w)$.
\medskip

\textbf{Property II.} For $v_1 < \cdots < v_r$, $\delta(v_1,v_r) = \max_{1 \leq j \leq r-1} \delta(v_j,v_{j+1})$.
\medskip

\begin{fact}\label{fact}
    Let $\{\delta_{i_j}\}_{j=1}^{s}\subset \{\delta_t\}_{t=1}^{r-1}$.
    If $\delta_{i_j}\neq\delta_{i_{j+1}}$ for all $j$, then any non-monotone sequence $\{\delta_{i_j}\}_{j=1}^{s}$ contains a local extremum.
\end{fact}
\medskip

Since $\delta_{i-1}\neq\delta_i$ for every $i$, every non-monotone sequence $\{\delta_i\}_{i=1}^{r-1}$ has a local extremum.
We will also use the following stepping-up properties, which are easy consequences of Properties I and II, see \cite{F-H-L-L,H-L-L-W-1,M-S-3}.
\medskip

\textbf{Property III.} For every 4-tuple $v_1 < \cdots < v_4$, if $\delta(v_1,v_2) > \delta(v_2,v_3)$, then $\delta(v_1,v_2) \neq \delta(v_3,v_4)$. Note that if $\delta(v_1,v_2) < \delta(v_2,v_3)$, it is possible that $\delta(v_1,v_2) = \delta(v_3,v_4)$.
\medskip

\textbf{Property IV.} For $v_1 < \cdots < v_r$, set $\delta_j = \delta(v_j,v_{j+1})$ for $j \in [r-1]$ and suppose that $\delta_1,\ldots,\delta_{r-1}$ forms a monotone sequence. Then for every subset of $k$ vertices $v_{i_1},v_{i_2},\ldots,v_{i_k}$ where $v_{i_1} < \cdots < v_{i_k}$, $\delta(v_{i_1},v_{i_2}),\delta(v_{i_2},v_{i_3}),\ldots,\delta(v_{i_{k-1}},v_{i_k})$ forms a monotone sequence. Moreover for every subset of $k-1$ such $\delta_j$'s, i.e. $\delta_{j_1},\delta_{j_2},\ldots,\delta_{j_{k-1}}$, there are $k$ vertices $v_{i_1},\ldots,v_{i_k}$ such that $\delta(v_{i_t},v_{i_{t+1}}) = \delta_{j_t}$.
\medskip

\section{The upper bound for $f^{(4)}_{5^-,6}(N)$}
The upper bound for $f^{(4)}_{5^-,6}(N)$ follows by applying a variant of the Erd\H{o}s-Hajnal stepping-up lemma starting from a graph to construct a $4$-graph. This process begins with a graph possessing a specific property, the existence of which is guaranteed by a direct application of the probabilistic method.
\begin{lemma}\label{phi-$H^4_4$-6}
For every $n\ge 5$, there exists an absolute constant $c>0$ such that there is a red/blue coloring $\phi$ of the pairs of $\{0,1,\ldots,\lfloor 2^{cn}\rfloor-1\}$ with the property that every $n$-set $A\subset \{0,1,\ldots,\lfloor 2^{cn}\rfloor-1\}$ contains a $6$-tuple $a_1< a_2< a_3< a_4< a_5< a_6$ satisfying 
$$\phi(a_1,a_4)=\phi(a_2,a_4)=\phi(a_3,a_4)=\phi(a_3,a_5)=\phi(a_4,a_5)=\text{red},$$
$$\phi(a_1,a_2)=\phi(a_1,a_3)=\phi(a_2,a_3)=\phi(a_3,a_6)=\phi(a_4,a_6)=\phi(a_5,a_6)=\text{blue}.$$

\end{lemma}

\noindent\emph{Proof of Lemma \ref{phi-$H^4_4$-6}.~} Set $D = \lfloor 2^{cn} \rfloor$, where $c$ is a sufficiently small constant that will be determined later. Consider the red/blue coloring $\phi$ of the pairs (edges) of $\{0,1,\ldots,D-1\}$, where each pair is colored red or blue independently with probability $1/2$ for each color.

We call a 6-tuple $a_1, a_2, a_3, a_4, a_5, a_6 \in \{0,1,\ldots,D-1\}$ with $a_1< a_2< a_3< a_4< a_5< a_6$ \emph{good} if 
$$\phi(a_1,a_4)=\phi(a_2,a_4)=\phi(a_3,a_4)=\phi(a_3,a_5)=\phi(a_4,a_5)=\text{red},$$
$$\phi(a_1,a_2)=\phi(a_1,a_3)=\phi(a_2,a_3)=\phi(a_3,a_6)=\phi(a_4,a_6)=\phi(a_5,a_6)=\text{blue},$$
and \emph{bad} otherwise.

Then for a fixed $6$-tuple, the probability that it is bad is $1-\frac{1}{2^{11}}=\frac{2047}{2048}$.
Now, let $A$ be a fixed $n$-subset of $\{0,1,\ldots,D-1\}$ and estimate the probability that it contains no good $6$-tuple.
Note that there exists a partial Steiner $(n,6,2)$-system $S$ \footnote{A 6-uniform hypergraph on an $n$-vertex set, with at least $c'n^2$ edges (for some constant $c'>0$), in which every pair of vertices is contained in at most one 6-tuple.}\cite{E-H-2}. Since the events for 6-tuples in $S$ are independent due to the fact that any two such $6$-tuples intersect in at most one vertex, the probability that all 6-tuples in $S$ are bad is at most $(2047/2048)^{c'n^2}$, and the probability that $A$ has no good 6-tuple is at most the probability that all 6-tuples in $S$ are bad. Therefore, the expected number of $n$-sets $A$ with only bad 6-tuples is at most
$$\binom{D}{n} \left(\frac{2047}{2048}\right)^{c'n^2} < \frac{1}{2},$$
again where we take $c$ sufficiently small. 
Thus the expected number of $n$-sets containing no good $6$-tuple is less than $1/2$. Hence there exists a coloring $\phi$ for which every $n$-set contains a good $6$-tuple. 
\hfill$\Box$

Let $c>0$ be the constant from Lemma \ref{phi-$H^4_4$-6}, and let $U=\{0,1,\ldots,\lfloor2^{cn}\rfloor-1\}$ and $\phi: \binom{U}{2}\rightarrow \{\text{red}, \text{blue}\}$ be a $2$-coloring of the pairs of $U$ satisfying the properties given in the lemma. Now, let $N=2^{\lfloor2^{cn}\rfloor}$ and $V(H)=\{0,1,\ldots,N-1\}$. We shall use the coloring $\phi$ to produce a $K^{(4)}_6$-free $4$-graph $H$ on $V(H)$ with $\alpha_{5^-}(H)<2^5n^5+1$ as follows. For any $4$-tuple $e=\langle v_1,v_2,v_3,v_4 \rangle$ of $V(H)$, set $e\in E(H)$ if and only if one of the following holds:
\begin{enumerate}
	\item[\textbf{(I)}] $\delta_1<\delta_2<\delta_3$, $\phi(\delta_1,\delta_2)=\text{red}$ and $\phi(\delta_1,\delta_3)=\phi(\delta_2,\delta_3)=\text{blue}$.
    \item[\textbf{(II)}] $\delta_1>\delta_2>\delta_3$,  $\phi(\delta_1,\delta_2)=\phi(\delta_1,\delta_3)=\text{red}$ and
    $\phi(\delta_2,\delta_3)=\text{blue}$.
    \item[\textbf{(III)}] $\delta_1>\delta_2<\delta_3$, $\delta_1< \delta_3$ and $\phi(\delta_1,\delta_3)=\phi(\delta_2,\delta_3)=\text{red}$. 
	\item[\textbf{(IV)}] $\delta_1>\delta_2<\delta_3$, $\delta_1> \delta_3$ and $\phi(\delta_2,\delta_3)=\text{blue}$.
\end{enumerate}

We first show that $H$ is $K_{6}^{(4)}$-free. 
Suppose to the contrary that there exists a set $P=\langle v_1,\ldots,v_{6} \rangle$ that induces a $K^{(4)}_{6}$ in $H$. We need the following claim.

\begin{claim}\label{no-mono4-6}
	There is no monotone subsequence $\{\delta'_{\ell}\}_{\ell=1}^{4}\subset \{\delta_i\}_{i=1}^{5}$ such that $\delta(u_1,\ldots,u_5)=\{\delta'_{1},\ldots,\delta'_{4}\}$ for some $\{u_1,\ldots,u_5\}\subset\{v_1,\ldots,v_{6}\}$.
\end{claim}

\noindent\emph{Proof of Claim \ref{no-mono4-6}.~}
Suppose, to the contrary, that there is a monotone subsequence $\{\delta'_{\ell}\}_{\ell=1}^{4}\subset \{\delta_i\}_{i=1}^{5}$ such that $\delta(u_1,\ldots,u_5)=\{\delta'_{1},\ldots,\delta'_{4}\}$ for some $\{u_1,\ldots,u_5\}\subset\{v_1,\ldots,v_{6}\}$. Note that $\delta(u_1,u_2,u_3,u_4)=(\delta'_1,\delta'_2,\delta'_3)$, so $\phi(\delta'_2,\delta'_3)=\text{blue}$ by \textbf{(I)} or \textbf{(II)}. But  $\delta(u_2,u_3,u_4,u_5)=(\delta'_2,\delta'_3,\delta'_4)$ implies 
$\phi(\delta'_2,\delta'_3)=\text{red}$, a contradiction.\hfill$\Box$

Now, we claim that there is no local maximum in $\delta(P)$. In fact, if there exists a local maximum $\delta_i$, then $(v_{i-1},v_i,v_{i+1},v_{i+2})$ is not an edge since $\delta(v_{i-1},v_i,v_{i+1},v_{i+2})=(\delta_{i-1}<\delta_i>\delta_{i+1})$ .
It follows from Claim \ref{no-mono4-6} that we have $\delta_1>\delta_2>\delta_3<\delta_4<\delta_5$. Note that $\delta_2\neq\delta_4$ by Property III. Since $\delta(v_3,v_4,v_5,v_6)=(\delta_3<\delta_4<\delta_5)$, $(v_3,v_4,v_5,v_6)$ is an edge implies $\phi(\delta_3,\delta_4)=\text{red}$ and $\phi(\delta_3,\delta_5)=\text{blue}$ by \textbf{(I)}. If $\delta_2>\delta_4$, then $(v_2,v_3,v_4,v_5)$ is an edge implies $\phi(\delta_3,\delta_4)=\text{blue}$ by \textbf{(IV)} since $\delta(v_2,v_3,v_4,v_5)=(\delta_2>\delta_3<\delta_4)$, a contradiction. If $\delta_2<\delta_4$, then we have $\delta_2<\delta_5$. Note that $\delta(v_2,v_3,v_4,v_6)=(\delta_2>\delta_3<\delta_5)$. $(v_2,v_3,v_4,v_6)$ is an edge implies $\phi(\delta_3,\delta_5)=\text{red}$ by \textbf{(III)}, a contradiction.

Thus, $H$ is $K^{(4)}_6$-free, as desired. 
\medskip

Now we show that $\alpha_{5^-}(H)<2^{5}n^{5}+1$. Suppose to the contrary that there are vertices $Q=\langle v_1,v_2,\cdots ,v_m\rangle$,  $m=2^{5}n^{5}+1$, which induce a $5^-$-independent set in $H$. Recall that $\delta_i=\delta(v_i,v_{i+1})$, and hence $\delta(Q)=\{\delta_1,\ldots, \delta_{m-1}\}$.

\begin{lemma}\label{no-mono-n}
There is no monotone subsequence $\{\delta_{i_j}\}_{j=1}^n\subset \{\delta_i\}_{i=1}^{m-1}$ such that for any $\{j_\ell:\ell\in[6]\}\subset[n]$ with $j_1< \cdots< j_6$ , there exists $\{u_1,\ldots,u_7\}\subset\{v_1,\ldots,v_m\}$ such that $\delta(u_1,\ldots,u_7)=\{\delta_{i_{j_1}},\ldots,\delta_{i_{j_6}}\}$.
\end{lemma}

\noindent\emph{Proof of Lemma \ref{no-mono-n}.~}Suppose to the contrary that $\{\delta_{i_j}\}_{j=1}^n$ is such a monotone subsequence. 

If $\{\delta_{i_j}\}_{j=1}^n$ is monotone increasing, then it follows from Lemma \ref{phi-$H^4_4$-6} that there is a $6$-tuple $\delta_{i_{j_1}},\ldots,\delta_{i_{j_6}}$ with $\delta_{i_{j_1}}<\cdots<\delta_{i_{j_6}}$ such that  
$$\phi(\delta_{i_{j_1}},\delta_{i_{j_4}})=\phi(\delta_{i_{j_2}},\delta_{i_{j_4}})=\phi(\delta_{i_{j_3}},\delta_{i_{j_4}})=\phi(\delta_{i_{j_3}},\delta_{i_{j_5}})=\phi(\delta_{i_{j_4}},\delta_{i_{j_5}})=\text{red},$$
$$\phi(\delta_{i_{j_1}},\delta_{i_{j_2}})=\phi(\delta_{i_{j_1}},\delta_{i_{j_3}})=\phi(\delta_{i_{j_2}},\delta_{i_{j_3}})=\phi(\delta_{i_{j_3}},\delta_{i_{j_6}})=\phi(\delta_{i_{j_4}},\delta_{i_{j_6}})=\phi(\delta_{i_{j_5}},\delta_{i_{j_6}})=\text{blue}.$$
Since there exists $\{u_1,\ldots,u_7\}\subset \{v_1,\ldots,v_m\}$ such that $\delta(u_1,\ldots,u_7)=\{\delta_{i_{j_1}},\ldots,\delta_{i_{j_6}}\}$ from the assumption, it follows from Property II that $\delta(u_4,u_5,u_6,u_7)=\delta(u_3,u_5,u_6,u_7)=(\delta_{i_{j_4}}<\delta_{i_{j_5}}<\delta_{i_{j_6}})$, $\delta(u_3,u_4,u_5,u_7)=(\delta_{i_{j_3}}<\delta_{i_{j_4}}<\delta_{i_{j_6}})$ and $\delta(u_3,u_4,u_6,u_7)=(\delta_{i_{j_3}}<\delta_{i_{j_5}}<\delta_{i_{j_6}})$. Thus, $(u_3,u_5,u_6,u_7)$, $(u_4,u_5,u_6,u_7)$, $(u_3,u_4,u_5,u_7)$ and $(u_3,u_4,u_6,u_7)$ are edges from \textbf{(I)}. Consequently, $\{u_3,\ldots,u_7\}$ contains a copy of $5^-$, a contradiction.

If $\{\delta_{i_j}\}_{j=1}^n$ is monotone decreasing, then it follows from Lemma \ref{phi-$H^4_4$-6} that there is a $6$-tuple $\delta_{i_{j_1}},\ldots,\delta_{i_{j_6}}$ with $\delta_{i_{j_1}}>\cdots>\delta_{i_{j_6}}$ such that  
$$\phi(\delta_{i_{j_3}},\delta_{i_{j_6}})=\phi(\delta_{i_{j_3}},\delta_{i_{j_5}})=\phi(\delta_{i_{j_3}},\delta_{i_{j_4}})=\phi(\delta_{i_{j_2}},\delta_{i_{j_3}})=\phi(\delta_{i_{j_2}},\delta_{i_{j_4}})=\text{red},$$
$$\phi(\delta_{i_{j_5}},\delta_{i_{j_6}})=\phi(\delta_{i_{j_4}},\delta_{i_{j_6}})=\phi(\delta_{i_{j_4}},\delta_{i_{j_5}})=\phi(\delta_{i_{j_1}},\delta_{i_{j_4}})=\phi(\delta_{i_{j_1}},\delta_{i_{j_3}})=\phi(\delta_{i_{j_1}},\delta_{i_{j_2}})=\text{blue}.$$
Since there exists $\{u_1,\ldots,u_7\}\subset \{v_1,\ldots,v_m\}$ such that $\delta(u_1,\ldots,u_7)=\{\delta_{i_{j_1}},\ldots,\delta_{i_{j_6}}\}$ from the assumption, $\delta(u_3,u_4,u_5,u_6)=\delta(u_3,u_4,u_5,u_7)=(\delta_{i_{j_3}}>\delta_{i_{j_4}}>\delta_{i_{j_5}})$, $\delta(u_3,u_5,u_6,u_7)=(\delta_{i_{j_3}}>\delta_{i_{j_5}}>\delta_{i_{j_6}})$ and $\delta(u_3,u_4,u_6,u_7)=(\delta_{i_{j_3}}>\delta_{i_{j_4}}>\delta_{i_{j_6}})$. Thus, $(u_3,u_4,u_5,u_6)$, $(u_3,u_4,u_5,u_7)$, $(u_3,u_5,u_6,u_7)$ and $(u_3,u_4,u_6,u_7)$ are edges from \textbf{(II)}. Consequently, $\{u_3,\ldots,u_7\}$ contains a copy of $5^-$, a contradiction.
\hfill$\Box$

\begin{claim}\label{rrr}
For any $\{u_1,\ldots,u_5\}\subset\{v_1,\ldots,v_{m}\}$ with $\delta(u_1,\ldots,u_5)=\{\delta'_{1},\ldots,\delta'_{4}\}$, none of the following occurs.

$(\Gamma_1)$. $\delta'_1>\delta'_2>\delta'_3<\delta'_4$ and $\delta'_1<\delta'_4$ and $\phi(\delta'_{1},\delta'_{4})=\phi(\delta'_{2},\delta'_{4})=\phi(\delta'_{3},\delta'_{4})=\text{red}$;

$(\Gamma_2)$. $\delta'_1>\delta'_2<\delta'_3<\delta'_4$ and $\delta'_1>\delta'_4$ and $\phi(\delta'_{2},\delta'_{4})=\phi(\delta'_{3},\delta'_{4})=\text{blue}$.
\end{claim}
\noindent\emph{Proof of Claim \ref{rrr}.~} For $(\Gamma_1)$, suppose to the contrary that there exists $\delta(u_1,\ldots,u_5)=(\delta'_{1}>\delta'_{2}>\delta'_{3}<\delta'_{4})$ such that $\delta'_1<\delta'_4$ and $\phi(\delta'_{1},\delta'_{4})=\phi(\delta'_{2},\delta'_{4})=\phi(\delta'_{3},\delta'_{4})=\text{red}$. By Property II, we have $\delta(u_1,u_2,u_3,u_5)=\delta(u_1,u_2,u_4,u_5)=(\delta'_1>\delta'_2<\delta'_4)$, $\delta(u_1,u_3,u_4,u_5)=(\delta'_1>\delta'_3<\delta'_4)$ and $\delta(u_2,u_3,u_4,u_5)=(\delta'_2>\delta'_3<\delta'_4)$. Together with $\phi(\delta'_{1},\delta'_{4})=\phi(\delta'_{2},\delta'_{4})=\phi(\delta'_{3},\delta'_{4})=\text{red}$ and \textbf{(III)}, this implies that $(u_1,u_2,u_3,u_5)$, $(u_1,u_2,u_4,u_5)$, $(u_1,u_3,u_4,u_5)$ and $(u_2,u_3,u_4,u_5)$ are edges. Consequently, $\{u_1,\ldots,u_5\}$ contains a copy of $5^-$, a contradiction.

For $(\Gamma_2)$, suppose to the contrary that there exists $\delta(u_1,\ldots,u_5)=(\delta'_{1}>\delta'_{2}<\delta'_{3}<\delta'_{4})$ such that $\delta'_1>\delta'_4$ and $\phi(\delta'_{2},\delta'_{4})=\phi(\delta'_{3},\delta'_{4})=\text{blue}$. By Property II, we obtain $\delta(u_1,u_2,u_4,u_5)=\delta(u_1,u_3,u_4,u_5)=(\delta'_1>\delta'_3<\delta'_4)$ and $\delta(u_1,u_2,u_3,u_5)=(\delta'_1>\delta'_2<\delta'_4)$. Together with $\phi(\delta'_{2},\delta'_{4})=\phi(\delta'_{3},\delta'_{4})=\text{blue}$ and \textbf{(IV)}, this implies that $(u_1,u_2,u_4,u_5)$, $(u_1,u_3,u_4,u_5)$ and $(u_1,u_2,u_3,u_5)$ are edges. Moreover, if $\phi(\delta'_{2},\delta'_{3})=\text{blue}$, 
then $(u_1,u_2,u_3,u_4)$ is an edge of $H$, since $\delta(u_1,u_2,u_3,u_4)=(\delta'_1>\delta'_2<\delta'_3)$ and \textbf{(IV)}; if $\phi(\delta'_{2},\delta'_{3})=\text{red}$, then $(u_2,u_3,u_4,u_5)$ is an edge of $H$, since $\delta(u_2,u_3,u_4,u_5)=(\delta'_2<\delta'_3<\delta'_4)$ and \textbf{(I)}. Thus, in either case, $\{u_1,\ldots,u_5\}$ contains a copy of $5^-$, a contradiction.\hfill$\Box$

Let $\beta_i=\frac{m-1}{(2n)^i}$, for $i\in[0,5]$. Since $m-1=(2n)^5$, each $\beta_i$ is an integer.
For $t\in[5]$, we will greedily construct \emph{$t$-layer local maxima sequences} $\Delta^{(t)}$ such that $\Delta^{(t)}\subset\Delta^{(t-1)}$, starting with $\Delta^{(0)}=\delta(Q)$, and satisfy the following property. (We do not require that the elements of $\Delta^{(t)}$ be distinct.)

\begin{itemize}
    \item[($\ast$)] For two consecutive elements $\delta_a$, $\delta_b\in\Delta^{(t)}$, we have $\delta_x<\max\{\delta_a ,\delta_b\}$ for all $a<x<b$, and hence $\delta_a\neq\delta_b$.
\end{itemize}

For $t\ge1$, assume now that we have obtained $\Delta^{(t-1)}$ satisfying the desired property. We restrict our attention to $\Delta^{(t-1)}$ and we will find $\Delta^{(t)}$ to be the first $\beta_t$ local maxima (with respect to  $\Delta^{(t-1)}$) as follows. For convenience, we abbreviate ``with respect to" as ``w.r.t." in the following discussion. We claim first that there is no monotone consecutive subsequence of length $n$. Otherwise, suppose such a subsequence $Q'$ exists. Without loss of generality, assume $Q'$ is increasing. For any $\delta_{j_1},\delta_{j_2},\ldots,\delta_{j_6}\in Q'$ with $j_1<\cdots<j_6$. Then $\delta(v_{j_1},v_{j_1+1},v_{j_2+1},v_{j_3+1},v_{j_4+1},v_{j_5+1},v_{j_6+1})=(\delta_{j_1}<\delta_{j_2}<\cdots<\delta_{j_6})$ by noting the first part of the property ($\ast$) of $\Delta^{(t-1)}$ and Property II, which contradicts Lemma \ref{no-mono-n}. It follows from the second part of the property ($\ast$) of $\Delta^{(t-1)}$ and Fact \ref{fact} that we can set $\Delta^{(t)}$ to be the first $\beta_t$ local maxima (w.r.t. $\Delta^{(t-1)}$). Therefore, $\Delta^{(t)}\subset \Delta^{(t-1)}$. 

To show the  property ($\ast$) for $\Delta^{(t)}$,  we consider two consecutive elements $\delta_a$, $\delta_b\in\Delta^{(t)}$ and  we may assume that $\delta_{a},\delta_{i_1},\delta_{i_2},\cdots,\delta_{i_j},\delta_{b}$ are consecutive elements in $\Delta^{(t-1)}$. Note that $\delta_{a}$ and $\delta_{b}$ are consecutive local maxima (w.r.t. $\Delta^{(t-1)}$), we have $\delta_{i_\ell}<\max\{\delta_a, \delta_b\}$ for $\ell\in[j]$. Furthermore, it follows from the inductive hypothesis that $\delta_x<\max\{\delta_{i_\ell},\delta_{i_{\ell+1}}\}$ for all $i_\ell<x<i_{\ell+1}$ and $\ell\in[j-1]$, then $\delta_x<\max\{\delta_{i_\ell},\delta_{i_{\ell+1}}\}<\max\{\delta_{a},\delta_{b}\}$. Thus, $\delta_x<\max\{\delta_{a},\delta_{b}\}$ for all $a<x<b$. Moreover, Property I implies that $\delta_a\neq\delta_b$, as desired. Otherwise $\delta(v_a,v_b)=\delta_a=\delta_b=\delta(v_b,v_{b+1})$, a contradiction.
\medskip

For $t\in[5]$ and $\delta_j\in\Delta^{(t)}\backslash\Delta^{(t+1)}$, where $\Delta^{(6)}=\emptyset$. Note that $\delta_j$ is a local maximum (w.r.t. $\Delta^{(t-1)}$), we always let $\delta_{j^-}$ and $\delta_{j^+}$ be the \textbf{closest} element to the left and right of $\delta_j$ in the sequence $\Delta^{(t-1)}$, respectively. Thus, 
$\delta_{j^-},\delta_{j^+}<\delta_{j}$. In particular, $\delta_{j^-},\delta_{j^+}\in \Delta^{(t-1)}\setminus\Delta^{(t)}$. From the above greedy construction, we can obtain the following observation by repeatedly using ($\ast$).

\begin{observation}\label{observation-1}
For $t\in [5]$ and  $\delta_j\in \Delta^{(t)}\backslash\Delta^{(t+1)}$, 
we have $\delta_x<\delta_{j}$ for each $x\in[j^-,j^+]\backslash \{j\}$.
\end{observation}

Note that $|\Delta^{(5)}|=\beta_{5}=1$. Choose $\delta_{a}\in \Delta^{(5)}$. Let $\delta_{b}:=\delta_{a^+}\in \Delta^{(4)}$ and $\delta_{c}:=\delta_{b^-}\in \Delta^{(3)}$, then we have $a<c<b$ and $\delta_{c}<\delta_{b}<\delta_{a}$. 

We claim that $\phi(\delta_{c},\delta_{b})=\text{red}$. Otherwise, we have $\phi(\delta_{c},\delta_{b})=\text{blue}$. Let $\delta_{d}:=\delta_{c^-}\in \Delta^{(2)}$, $\delta_{e}:=\delta_{d^+}\in \Delta^{(1)}$ and $\delta_{f}:=\delta_{e^+}\in\delta(Q)$. It follows from Observation \ref{observation-1} that we have $\delta(v_{a},v_{d},v_{d+1},v_{c+1},v_{b+1})=(\delta_{a}>\delta_{d}<\delta_{c}<\delta_{b})$. 
By Claim \ref{rrr} $(\Gamma_2)$, we must have $\phi(\delta_d,\delta_b)=\text{red}$. 
Similarly, $\delta(v_{a},v_{e},v_{e+1},v_{c+1},v_{b+1})=(\delta_{a}>\delta_{e}<\delta_{c}<\delta_{b})$ and $\delta(v_{a},v_{f},v_{f+1},v_{c+1},v_{b+1})=(\delta_{a}>\delta_{f}<\delta_{c}<\delta_{b})$ imply $\phi(\delta_e,\delta_b)=\phi(\delta_f,\delta_b)=\text{red}$. Since $\delta(v_{d},v_{e},v_{f},v_{f+1},v_{b+1})=(\delta_{d}>\delta_{e}>\delta_{f}<\delta_{b})$ and $\delta_{d}<\delta_{b}$, we have a contradiction with Claim \ref{rrr} $(\Gamma_1)$.

Thus, we have $\phi(\delta_{c},\delta_{b})=\text{red}$. Let $\delta_{d}:=\delta_{c^+}\in \Delta^{(2)}$. 
If $\phi(\delta_{d},\delta_{b})=\text{blue}$, then we let $\delta_{e}:=\delta_{d^-}\in \Delta^{(1)}$ and $\delta_{f}:=\delta_{e^+}\in\delta(Q)$. Thus, we have $\delta(v_{a},v_{e},v_{e+1},v_{d+1},v_{b+1})=(\delta_{a}>\delta_{e}<\delta_{d}<\delta_{b})$. By Claim \ref{rrr} $(\Gamma_2)$, we have $\phi(\delta_e,\delta_b)=\text{red}$.
Similarly, $\delta(v_{a},v_{f},v_{f+1},v_{d+1},v_{b+1})=(\delta_{a}>\delta_{f}<\delta_{d}<\delta_{b})$ implies $\phi(\delta_f,\delta_b)=\text{red}$. Then $\delta(v_{c},v_{e},v_{f},v_{f+1},v_{b+1})=(\delta_{c}>\delta_{e}>\delta_{f}<\delta_{b})$, contradicting Claim \ref{rrr} $(\Gamma_1)$.

Hence, we have $\phi(\delta_{d},\delta_{b})=\text{red}$. Let $\delta_{e}:=\delta_{d^+}\in \Delta^{(1)}$ and $\delta_{f}:=\delta_{e^-}\in\delta(Q)$. Thus, we have $\delta(v_{a},v_{f},v_{f+1},v_{e+1},v_{b+1})=(\delta_{a}>\delta_{f}<\delta_{e}<\delta_{b})$. 
It follows from Claim \ref{rrr} $(\Gamma_2)$ that $\phi(\delta_f,\delta_b)=\text{red}$ or $\phi(\delta_e,\delta_b)=\text{red}$. For the former case, note that  $\phi(\delta_{c},\delta_{b})=\phi(\delta_{d},\delta_{b})=\text{red}$ and $\delta(v_{c},v_{d},v_{f},v_{f+1},v_{b+1})=(\delta_{c}>\delta_{d}>\delta_{f}<\delta_{b})$, which contradicts Claim \ref{rrr} $(\Gamma_1)$. For the latter case, note that  $\phi(\delta_{c},\delta_{b})=\phi(\delta_{d},\delta_{b})=\text{red}$ and $\delta(v_{c},v_{d},v_{e},v_{e+1},v_{b+1})=(\delta_{c}>\delta_{d}>\delta_{e}<\delta_{b})$, which contradicts Claim \ref{rrr} $(\Gamma_1)$ again.

Thus, $\alpha_{5^-}(H)<2^5n^5+1$. This completes the proof of the upper bound for $f^{(4)}_{5^-,6}(N)$.

\section{The upper bound for $f^{(4)}_{H^4_2,6}(N)$}
Again, we use a variant of the Erd\H{o}s-Hajnal stepping-up lemma to construct a $4$-graph from a graph, thereby establishing a new upper bound for $f^{(4)}_{H^4_2,6}(N)$. We use the following lemma and omit its proof since it follows by a similar probabilistic argument used for Lemma \ref{phi-$H^4_4$-6}.

\begin{lemma}\label{phi-$H^4_2$}
	For $n\ge 5$, there is an absolute constant $c>0$ such that there exists a red/blue coloring $\phi$ of the pairs of $\{0,1,\ldots,\lfloor2^{cn}\rfloor-1\}$ such that every $n$-set $A\subset \{0,1,\ldots,\lfloor 2^{cn}\rfloor-1\}$ contains a $4$-tuple $a_i< a_j< a_k< a_\ell$ satisfying $$\phi(a_i,a_k)=\phi(a_j,a_\ell)=\text{red},~~~\phi(a_i,a_j)=\phi(a_j,a_k)=\phi(a_k,a_\ell)=\text{blue}.$$
	
\end{lemma}

Let $c>0$ be the constant from Lemma \ref{phi-$H^4_2$}, define $U=\{0,1,\ldots,\lfloor2^{cn}\rfloor-1\}$ and $\phi: \binom{U}{2}\rightarrow \{\text{red}, \text{blue}\}$ be a $2$-coloring of the pairs of $U$ satisfying the properties given in the lemma. Now, let $N=2^{\lfloor2^{cn}\rfloor}$ and $V(H)=\{0,1,\ldots,N-1\}$. We shall use the coloring $\phi$ to produce a $K^{(4)}_6$-free $4$-graph $H$ on $V(H)$ with $\alpha_{H^4_2}(H)<n^2$ as follows. For any $4$-tuple $e=\langle v_1,v_2,v_3,v_4 \rangle$ of $V(H)$, set $e\in E(H)$ if and only if one of the following holds:
\begin{enumerate}
	\item[\textbf{(i)}] $\delta_1,\delta_2,\delta_3$ form a monotone sequence, $\phi(\delta_1,\delta_3)=\text{red}$ and $\phi(\delta_1,\delta_2)=\phi(\delta_2,\delta_3)=\text{blue}$.
	\item[\textbf{(ii)}] $\delta_1>\delta_2<\delta_3$, $\delta_1< \delta_3$ and $\phi(\delta_1,\delta_3)=\text{red}$.
	\item[\textbf{(iii)}] $\delta_1<\delta_2>\delta_3$ and $\phi(\delta_1,\delta_2)=\text{blue}$.
\end{enumerate}

We first show that $H$ is $K_{6}^{(4)}$-free. Suppose to the contrary that there exists a set $P=\langle v_1,\ldots,v_{6} \rangle$ inducing a $K^{(4)}_{6}$ in $H$. To reach a contradiction, we need the following claim.

\begin{claim}\label{no-mono2}
	There is no monotone subsequence $\{\delta'_{\ell}\}_{\ell=1}^{4}\subset \{\delta_i\}_{i=1}^{5}$ such that $\delta(u_1,\ldots,u_5)=\{\delta'_{1},\ldots,\delta'_{4}\}$ for some $\{u_1,\ldots,u_5\}\subset\{v_1,\ldots,v_{6}\}$.
\end{claim}
\noindent\emph{Proof of Claim \ref{no-mono2}.~}Suppose, to the contrary, that $\{\delta'_{\ell}\}_{\ell=1}^4$
is a monotone increasing subsequence, without loss of generality, and that there exists $\{u_1,\ldots,u_5\}\subset\{v_i\}_{i=1}^{6} $ such that $\delta(u_1,\ldots,u_5)=(\delta'_1<\delta'_2<\delta'_3<\delta'_4)$.
Since $\delta(u_1,u_2,u_3,u_4)=(\delta'_1<\delta'_2<\delta'_3)$, we have $\phi(\delta'_1,\delta'_3)=\text{red}$ from \textbf{(i)}, which implies that $(u_1,u_2,u_4,u_5)\notin E(H)$ since $\delta(u_1,u_2,u_4,u_5)=(\delta'_1<\delta'_3<\delta'_4)$ and $\phi(\delta'_1,\delta'_3)=\text{red}$, a contradiction.
\hfill$\Box$

Let $\delta_{k} = \delta(v_{k}, v_{k+1})$ be the unique largest element in $\{\delta_i\}_{i=1}^{5}$. The uniqueness of $\delta_{k}$ follows from Property III. 
We claim that $k\ge 3$. Otherwise, $k\le2$. We consider $\{\delta_{k+1},\delta_{k+2},\delta_{k+3}\}\subset \{\delta_i\}_{i=1}^{5}$. If $\delta_{k+1}>\delta_{k+2}>\delta_{k+3}$, then $\{\delta_{k},\delta_{k+1},\delta_{k+2},\delta_{k+3}\}$ is a monotone decreasing subsequence, which contradicts Claim \ref{no-mono2}. So we have $\delta_{k+1}<\delta_{k+2}$ or $\delta_{k+2}<\delta_{k+3}$  by noting Property I, without loss of generality, we may assume that $\delta_{k+1}<\delta_{k+2}$. Since $$\delta(v_{k},v_{k+1},v_{k+2},v_{k+3})=(\delta_k>\delta_{k+1}<\delta_{k+2}) ~\text{and}~ \delta_k>\delta_{k+2},$$ $(v_{k},v_{k+1},v_{k+2},v_{k+3})\notin E(H)$ by noting \textbf{(ii)}, a contradiction. 

We now shift our attention to the sequence 
$\{\delta_i\}_{i=1}^{k-1}$. We claim that there is no local extremum in $\{\delta_i\}_{i=1}^{k-1}$. Otherwise, suppose first that there exists an index $a$ such that $\delta_{a-1}<\delta_{a}>\delta_{a+1}$ in the sequence 
$\{\delta_i\}_{i=1}^{k-1}$. Since $(v_{a-1},v_{a},v_{a+1},v_{k+1})\in E(H)$ and $\delta(v_{a-1},v_{a},v_{a+1},v_{k+1})=(\delta_{a-1}<\delta_{a}<\delta_{k})$,  it follows from \textbf{(i)} that $\phi(\delta_{a},\delta_k)=\text{blue}$. Then, by \textbf{(ii)}, we have $(v_{a},v_{a+1},v_{a+2},v_{k+1})\notin E(H)$, because $\delta(v_{a},v_{a+1},v_{a+2},v_{k+1})=(\delta_{a}>\delta_{a+1}<\delta_{k})~\text{and}~ \delta_a<\delta_{k}$, a contradiction. Suppose now that there exists an index $a$ such that $\delta_{a-1}>\delta_{a}<\delta_{a+1}$ in the sequence $\{\delta_i\}_{i=1}^{k-1}$.  
Since $(v_{a-1},v_{a},v_{a+1},v_{a+2})\in E(H)$ and $\delta(v_{a-1},v_{a},v_{a+1},v_{a+2})=(\delta_{a-1}>\delta_{a}<\delta_{a+1})$, we have $\delta_{a-1}<\delta_{a+1}$ and $\phi(\delta_{a-1},\delta_{a+1})=\text{red}$ by \textbf{(ii)}. Then, by \textbf{(i)} and $\delta(v_{a-1},v_{a},v_{a+2},v_{k+1})=(\delta_{a-1}<\delta_{a+1}<\delta_{k})$, we have $(v_{a-1},v_{a},v_{a+2},v_{k+1})\notin E(H)$, again a contradiction.
Thus, $\{\delta_i\}_{i=1}^{k-1}$ is a monotone sequence. Note that 
$\{\delta_1,\delta_2,\delta_3,\delta_4\}$ is not a monotone sequence by Claim \ref{no-mono2}. Then, we have $k\le4$ and thus, $3\le k\le4$.  Since $\delta(v_1,v_2,v_{k+1},v_{k+2})=(\delta_1<\delta_k>\delta_{k+1})$, it follows from \textbf{(iii)} that $\phi(\delta_1,\delta_k)=\text{blue}$. However, $(v_1,v_2,v_3,v_{k+1})\in E(H)$  implies $\phi(\delta_1,\delta_k)=\text{red}$ by \textbf{(i)} or \textbf{(ii)}, a contradiction.

Thus, $H$ is $K^{(4)}_6$-free, as we need. 
\medskip

Set  $m=n^2$. It remains to show that $\alpha_{H^4_2}(H)<m$. On the contrary, there are vertices $Q=\langle v_1,v_2,\cdots ,v_m\rangle$ that induce a $H^4_2$-independent set in $H$. Recall that $\delta_i=\delta(v_i,v_{i+1})$ and $\delta(Q)=\{\delta_i\}_{i=1}^{m-1}$. We will also begin by noting the following lemma and omit its proof since it follows by a similar argument to that in Lemma \ref{no-mono-n}.
	\begin{lemma}\label{bad-$H^4_2$}
	There is no monotone subsequence $\{\delta_{i_j}\}_{j=1}^n\subset \{\delta_i\}_{i=1}^{m-1}$ such that for any $a,b,c,d \in [n]$ with $a<b<c<d$, there exists $\{u_1,\ldots,u_5\}\subset\{v_1,\ldots,v_m\}$ such that $\delta(u_1,\ldots,u_5)=\{\delta_{i_a},\delta_{i_b},\delta_{i_c},\delta_{i_d}\}$.
\end{lemma}

We claim that there is no integer $j \in [m-n]$ such that the sequence $\{\delta_i\}_{i=j}^{j+n-1}$ is monotone. Otherwise, for any length 4 monotone subsequence $\{\delta_{i_1}, \delta_{i_2}, \delta_{i_3}, \delta_{i_4}\} \subset \{\delta_i\}_{i=j}^{j+n-1}$, there exists $ \{u_1,\ldots,u_5\}\subset \{v_1, \ldots, v_m\}$ such that $\delta(u_1,\ldots,u_5)=\{\delta_{i_1}, \delta_{i_2}, \delta_{i_3}, \delta_{i_4}\}$ by Property IV, contradicting Lemma \ref{bad-$H^4_2$}. 

Next, we greedily construct the following sets $L_r$, $R_r$, $S_r${ for $r\le 2n-3$.} Start with $L_0=R_0=\emptyset$, $\sigma_0=0$, $\tau_0=n^2$ and $S_0=\{1,2,\cdots,n^2-1\}$. At each step $r$,
\begin{enumerate}
	\item $\sigma_r=0$ if $L_r$ is empty and $\sigma_r=\max(L_r)$  otherwise. Similarly, $\tau_r=n^2$ if $R_r$ is empty and $\tau_r=\min(R_r)$  otherwise.
	\item  $\delta_a>\delta_s$ for $a\in L_r$ and $s\in S_r$. Similarly, $\delta_b>\delta_s$ for $b\in R_r$ and $s\in S_r$.
    \item $S_r=\{j:\sigma_r<j<\tau_r\}$ and $\tau_r-\sigma_r\ge n^2-1-(n-1)|L_r|-|R_r|$. 
    \item $|L_r|+|R_r|=r$ and $|L_r|,|R_r|< n-1$  for $r<2n-3$.
\end{enumerate}
Note that these properties hold for $r=0$ by definition. 
Now assume that for some $r<2n-3$, 
we have constructed $L_r$, $R_r$, $S_r$ satisfying the desired properties, we shall construct $L_{r+1}$, $R_{r+1}$, $S_{r+1}$ and define $\delta_{i_{r+1}}$ as follows.

Let $\delta_{i_{r+1}}=\max_{s\in S_r} \delta_{s}$ be the unique largest element by noting $\tau_r-\sigma_r\ge n^2-1-(n-1)(n-2)-(n-2)>0$
, and the uniqueness of $\delta_{i_{r+1}}$ follows from Property III. 
We claim that either $i_{r+1}-\sigma_r\le n-1$ or $\tau_r-i_{r+1}= 1$.  Otherwise, $\sigma_r+n\le i_{r+1}<\tau_r-1$, then there exists $a<i_{r+1}-1$ in $S_r$ such that $\delta_a>\delta_{a+1}$  by Lemma \ref{bad-$H^4_2$} and we can take $b=i_{r+1}+1$ in $S_r$. 
Note that $$\delta(v_a,v_{a+1},v_{a+2},v_b,v_{b+1})=(\delta_a>\delta_{a+1}<\delta_{i_{r+1}}>\delta_{b})~\text{and}~\delta_a<\delta_{i_{r+1}}.$$
If $\phi(\delta_{a},\delta_{i_{r+1}})=\text{red}$,  then by \textbf{(ii)}, we have $(v_a,v_{a+1},v_{a+2},v_b),(v_a,v_{a+1},v_{a+2},v_{b+1})\in E(H)$. If $\phi(\delta_{a},\delta_{i_{r+1}})=\text{blue}$, then by  \textbf{(iii)}, we have $(v_a,v_{a+1},v_b,v_{b+1}),(v_a,v_{a+2},v_b,v_{b+1})\in E(H)$.	
Thus, $\{v_a,v_{a+1},v_{a+2},v_b,v_{b+1}\}$ always contains a copy of  $H^4_2$, a contradiction. 

Therefore, we have  $i_{r+1}-\sigma_r\le n-1$ or $\tau_r-i_{r+1}= 1$. If $i_{r+1}-\sigma_r\le n-1$, then set
$L_{r+1} = L_r\cup\{i_{r+1}\}$, $R_{r+1}=R_r$, $\sigma_{r+1}=i_{r+1}$, $\tau_{r+1}=\tau_r$ and $S_{r+1}=\{j:\sigma_{r+1}<j<\tau_{r+1}\}$, 
then $\tau_{r+1}-\sigma_{r+1}= \tau_{r}-i_{r+1}\ge \tau_{r}-(n-1+\sigma_{r})\ge n^2-1-(n-1)(|L_r|+1)-|R_r|=n^2-1-(n-1)|L_{r+1}|-|R_{r+1}|$. If $\tau_r-i_{r+1}= 1$, then set
$L_{r+1} = L_r$, $R_{r+1}=R_r\cup\{i_{r+1}\}$, $\sigma_{r+1}=\sigma_r$, $\tau_{r+1}=i_{r+1}$ and $S_{r+1}=\{j:\sigma_{r+1}<j<\tau_{r+1}\}$, 
then $\tau_{r+1}-\sigma_{r+1}= i_{r+1}-\sigma_{r}= (\tau_{r}-1)-\sigma_{r}\ge n^2-1-(n-1)|L_r|-(|R_r|+1)=n^2-1-(n-1)|L_{r+1}|-|R_{r+1}|$. 
Thus, in any case, the first three properties and the first part of the fourth property hold by definition. 

Now we consider the second part of the fourth property. Suppose to the contrary that $|L_{r+1}|=n-1$ without loss of generality. Since $|S_{r+1}|= \tau_{r+1}-\sigma_{r+1}\ge n^2-(n-1)(n-1)-(n-1)>0$, we can take $i_{r+2}\in S_{r+1}$. Let $L_{r+1}\cup \{{i_{r+2}}\}=\{\ell_1,\ldots,\ell_n\}$ where $\ell_1<\ell_2<\cdots<\ell_n$, then we have $\delta_{{\ell_1}} >\delta_{{\ell_2}}>\cdots > \delta_{{\ell_n}}$. 
In particular, Property II implies that for any $a, b, c, d \in [n]$ and $a<b<c<d$,
$\delta(v_{{\ell_a}}, v_{{\ell_b}}, v_{{\ell_c}}, v_{{\ell_d}}, v_{{\ell_{d}}+1})=(\delta_{{\ell_a}} >\delta_{{\ell_b}}>\delta_{{\ell_c}} >\delta_{{\ell_d}} )$, which is in contradiction with Lemma \ref{bad-$H^4_2$}. Thus, the second part of fourth property holds as desired. 

We can construct these sets as long as $r< 2n-3$. Now, consider $r=2n-3$. Since $|L_{2n-3}|+|R_{2n-3}|=2n-3$, we have either $|L_{2n-3}| = n-1$ or $|R_{2n-3}| = n-1$. By a similar argument as above, we can always find a monotone sequence $\{\delta_{\ell_i}\}_{i=1}^n\subset \{\delta_i\}_{i=1}^{m-1}$ such that for any $a,b,c,d\in[n]$ with $a<b<c<d$, there exists a set $\{u_1,\ldots,u_5\}\subset \{v_1,\ldots ,v_m\}$ with $\delta(u_1,\ldots,u_5)=\{\delta_{\ell_a},\delta_{\ell_b},\delta_{\ell_c},\delta_{\ell_d}\}$. This contradicts Lemma \ref{bad-$H^4_2$}.
 
Thus, $\alpha_{H^4_2}(H)<n^2$. This completes the proof of the upper bound for $f^{(4)}_{H^4_2,6}(N)$.

\end{spacing}

\begin{thebibliography}{99}
\bibitem{A-K-S-1}
M. Ajtai, J. Koml\'{o}s, and E. Szemer\'{e}di, A note on Ramsey numbers, \emph{J. Combin. Theory Ser. A.}
29 (1980), 354--360.
\bibitem{A-K-S-2}
M. Ajtai, J. Koml\'{o}s, and E. Szemer\'{e}di, A dense infinite Sidon sequence, \emph{European J. Combin.} 2 (1981), 1--11.
\bibitem{ak}
 N. Alon and M. Krivelevich, Constructive bounds for a Ramsey-type problem, \emph{ Graphs Combin.} 13
(1997), 217--225.
\bibitem{B-C-L-1}
J. Balogh, C. Chen, and H. Luo, On the maximum $F$-free induced subgraphs in $K_t$-free graphs, \emph{Random Struct. Algorithms.} 66 (1), 2025.
\bibitem{B-K-2}
T. Bohman and P. Keevash, The early evolution of the $H$-free process, \emph{Invent. Math.} 181 (2010), 291--336.
\bibitem{bh}
 B. Bollob\'{a}s and H. R. Hind, Graphs without large triangle free subgraphs, \emph{ Discrete Math.} 87 (1991), 119--131.
\bibitem{C-G-M-S}
M. Campos, S. Griffiths, R. Morris and J. Sahasrabudhe,
\emph{An exponential improvement for diagonal Ramsey},
accepted for publication in \emph{Ann. of Math.}, to appear.
 \bibitem{C-J-M-S}
M. Campos, M. Jenssen, M. Michelen, and J. Sahasrabudhe, A new lower bound for the Ramsey
numbers $R(3, k)$, arXiv preprint \emph{arXiv:2505.13371}, 2025.
\bibitem{C-G-E}
F. R. K. Chung and R. L. Graham: Erd\H{o}s on Graphs: His Legacy of Unsolved Problems, A. K. Peters Ltd., Wellesley, MA, 1998.
\bibitem{cfs}
D. Conlon, J. Fox and B. Sudakov, An improved bound for the stepping-up lemma, \emph{Discrete Appl. Math.} 161 (2013), 1191--1196.
\bibitem{C-F-S-3}
D. Conlon, J. Fox and B. Sudakov, Hypergraph Ramsey numbers, \emph{J. Amer. Math. Soc.} 23 (2010), 247--266.
\bibitem{C-F-S-2}
D. Conlon, J. Fox and B. Sudakov, Short proofs of some extremal results,\emph{ Combin. Probab. Comput.} 23 (2014), 8--28.
\bibitem{C-F-S-4}
D. Conlon, J. Fox, and B. Sudakov, Recent developments in graph Ramsey theory, in \emph{Surveys in Combinatorics 2015}, London Math. Soc. Lecture Note Ser., Cambridge University Press, Cambridge, 2015, pp. 49--118.
\bibitem{D-H-L-W}
L. Du, X. Hu, R. Liu and G. Wang, A step towards the Erd\H{o}s-Rogers problem, arXiv preprint \emph{arXiv:2603.12610}, 2026.
\bibitem{D-M}
A. Dudek and D. Mubayi, On generalized Ramsey numbers for 3-uniform hypergraphs, \emph{J. Graph Theory}
76 (2014), 217--223.
\bibitem{D-R-R}
A. Dudek, T. Retter and V. R\'{o}dl, On generalized Ramsey numbers of Erd\H{o}s and Rogers, \emph{J. Combin. Theory Ser. B} 109 (2014), 213--227.
\bibitem{dr}
 A. Dudek and V. R\"{o}dl, On $K_s$-free subgraphs in $K_{s+k}$-free graphs and vertex Folkman numbers, \emph{ Combinatorica} 31 (2011), 39--53.
\bibitem{E-H-Con}
P. Erd\H{o}s and A. Hajnal, On Ramsey like theorems, problems and results, Combinatorics
(Proc. Conf. Combinatorial Math., Math. Inst., Oxford, 1972), 123--140, Inst. Math. Appl.,
Southend-on-Sea, 1972.
\bibitem{E-H-2}
P. Erd\H{o}s and H. Hanani, On a limit theorem in combinatorical analysis, \emph{Publ. Math Debrecen} 10 (1963), 10--13.
\bibitem{E-R-2}
P. Erd\H{o}s and R. Rado, Combinatorial theorems on classifications of subsets of a given set, \emph{P. London Math. Soc.} (3) 2 (1952), 417--439.
\bibitem{E-R-1}
P. Erd\H{o}s and C. A. Rogers, The construction of certain graphs, \emph{Canad. J. Math.} 14 (1962), 702--707.
\bibitem{F-H-L-L}
C. Fan, X. Hu, Q. Lin and X. Lu, New bounds of two hypergraph Ramsey problems, arXiv preprint \emph{arXiv:2410.22019}, 2024.
\bibitem{G-J-S-1}
L. Gishboliner, O. Janzer, and B. Sudakov, Induced subgraphs of $K_r$-free graphs and the Erd\H{o}s–Rogers problem, \emph{Combinatorica} 45 (2025).
\bibitem{GJ20}
W.~T. Gowers and O.~Janzer, Improved bounds for the {E}rd{\H o}s-{R}ogers function,
 \emph{ Adv. Combin.} 3 (2020), 27pp.
\bibitem{G-R-S-1}
R. L. Graham, B. L. Rothschild and J. H. Spencer, Ramsey Theory, 2nd edn, \emph{Wiley Interscience Series in Discrete Mathematics and Optimization } (Wiley, New York, 1990).
\bibitem{G-N-N-W}
P. Gupta, N. Ndiaye, S. Norin, and L. Wei, Optimizing the CGMS upper bound on Ramsey numbers, arXiv preprint \emph{arXiv:2407.19026}, 2024.
\bibitem{H-N-1}
X. He and J. Nie,
Generalized Erd\H{o}s--Rogers problems for hypergraphs,
\emph{European J. Combin.} 135 (2026), 104372.
\bibitem{H-H-K-P}
Z. Hefty, P. Horn, D. King and F. Pfender, Improving $R(3,k)$ in just two bites, arXiv preprint \emph{arXiv:2510.19718}, 2025.
\bibitem{H-L}
X. Hu and Q. Lin, Ramsey numbers and a general Erd\H{o}s-Rogers function, \emph{Discrete Math.} 347 (2024), 114203.
\bibitem{H-L-L-W-1}
X. Hu, Q. Lin, X. Lu and G. Wang, Phase transitions of the Erd\H{o}s-Gy\'{a}rf\'{a}s function, arXiv preprint \emph{arXiv:2504.05647}, 2025.
\bibitem{H-M-S}
Z. Hunter, A. Milojevi\'{c} and B. Sudakov, Gaussian random graphs and Ramsey numbers, arXiv preprint \emph{arXiv:2512.17718}, 2025.
\bibitem{Janzer}
O. Janzer and B.~Sudakov, Improved bounds for the Erd\H{o}s-Rogers $(s,s+2)$-problem, \emph{ Random Struct. Algorithms.} 66 (2025), 5 pp.
\bibitem{K-1}
 J. H. Kim. The Ramsey number $R(3, t)$ has order of magnitude $t^2/ \log t$, \emph{ Random Struct. Algorithms.} 7 (1995), 173--207.
\bibitem{kri-1}
 M. Krivelevich, $K_s$-free graphs without large $K_r$-free subgraphs, \emph{ Combin. Probab. Comput.} 3 (1994),
349--354.
\bibitem{kri-2} M. Krivelevich, Bounding Ramsey numbers through large deviation inequalities, \emph{ Random Struct. Algorithms.} 7 (1995), 145--155.
\bibitem{M-S-X}
J. Ma, W. Shen and S. Xie, An exponential improvement for Ramsey lower bounds, arXiv preprint \emph{arXiv:2507.12926}, 2025.
\bibitem{M-V-2}
S. Mattheus and J. Verstra\"{e}te, The asymptotics of $r(4,t)$, \emph{Ann. of Math.} 199 (2024), 919--941.
\bibitem{Mor}
R. Morris, Some recent results in Ramsey theory, arXiv preprint \emph{arXiv:2601.05221,} 2026.
\bibitem{L-R-Z}
Y. Li, C. C. Rousseau, W. Zang, An upper bound for Ramsey numbers, \emph{Appl. Math. Lett.} 17 (2004), 663--665.
\bibitem{M-S-1}
D. Mubayi and A. Suk, A survey of hypergraph Ramsey problems, \emph{Discrete Mathematics and Applications} (eds A. Raigorodskii and M. T. Rassias; Springer, Cham, 2020).
\bibitem{M-S-3}
D. Mubayi and A. Suk, New lower bounds for hypergraph Ramsey numbers, \emph{Bull. London Math. Soc.} 50 (2018), 189--201.
\bibitem{M-S-4}
D. Mubayi and A. Suk, Off-diagonal hypergraph Ramsey numbers, \emph{J. Combin. Theory Ser. B} 125 (2017), 168--177.
\bibitem{M-S-5}
D. Mubayi and A. Suk, Constructions in Ramsey theory, \emph{J. London Math. Soc.} (2) 97 (2018), 247--257.
\bibitem{M-V-1-1}
D. Mubayi and J. Verstra\"{e}te, Erd\H{o}s–Rogers functions for arbitrary pairs of graphs, arXiv preprint \emph{arXiv:2407.03121,} 2024.
\bibitem{R}
F.~P.~Ramsey, On a problem of formal logic,  \emph{ Proc.~London Math.~Soc.} 30 (1929), 264--286.
\bibitem{S-1}
J. B. Shearer, A note on the independence number of triangle-free graphs, \emph{Discrete Math.} 46 (1983), 83--87.
\bibitem{Sud-1}
B. Sudakov, Large $K_r$-free subgraphs in $K_s$-free graphs and some other Ramsey-type problems,
\emph{Random Struct. Algorithms.} 26 (2005), no. 3, 253--265.
\bibitem{Sud-2}
B. Sudakov, A new lower bound for a Ramsey-type problem, \emph{Combinatorica} 25 (2005), no. 4,
487--498.
\bibitem{W-1}
G. Wolfovitz, $K_4$-free graphs without large induced triangle-free subgraphs, \emph{Combinatorica} 33 (2013), 623--631.
\end{thebibliography}
\end{document}